\documentclass[12pt]{amsart}

\usepackage{euscript}
\usepackage{amsmath}
\usepackage{amsthm}
\usepackage{epsfig}
\usepackage{amssymb}\usepackage{amscd}
\usepackage{epic}
\usepackage{eufrak}
\usepackage{marginnote}
\usepackage{color}
\usepackage{enumerate}
\usepackage{url}

\numberwithin{equation}{section}
\numberwithin{equation}{subsection}

\theoremstyle{plain}
\newtheorem{theorem}[equation]{Theorem}

\theoremstyle{definition}

\newtheorem{remark}[equation]{Remark}

\numberwithin{equation}{section}
\numberwithin{equation}{subsection}

\newcommand{\labelpar}{\label}

\title{Classification of rational unicuspidal curves with two Newton pairs}

\author{J\'{o}zsef Bodn\'{a}r}
\address{A. R\'enyi Institute of Mathematics, 1053 Budapest,
Re\'altanoda u. 13-15,  Hungary.}
\email{bodnar.jozef@renyi.mta.hu}
\thanks{The author is supported by
 the `Lend\"ulet' and ERC program `LTDBud' at MTA Alfr\'ed R\'enyi Institute of Mathematics.}
\keywords{rational cuspidal curves, Newton pairs, plane curve singularities}

\date{}

\begin{document}

\begin{abstract}
 Based on Tiankai Liu's PhD thesis \cite{Liu}, we give a complete classification of local topological types of singularities with two Newton pairs on rational unicuspidal complex projective plane curves.
\end{abstract}

\maketitle


\pagestyle{myheadings} \markboth{{\normalsize
J. Bodn\'ar}}{ {\normalsize Classification of rational unicuspidal curves with two Newton pairs}}

\section{Intorduction}

\subsection{Introduction}

We give a complete list of possible local topological types of plane curve singularities of rational unicuspidal complex projective curves whose only singularity has exactly two Newton pairs. An analogous list for rational unicuspidal curves with \emph{one} Newton (or Puiseux) pair was presented in \cite{BLMN1p}.

In \cite{BLMN1}, a conjecture was formulated by J. Fern\'andez de Bobadilla, I. Luengo, A. Melle-Hern\'andez and A. N\'{e}methi on the local topological types of rational cuspidal curves. Moreover, the conjecture was checked explicitly for all cuspidal curves whose complement has logarithmic Kodaira dimension $\overline{\kappa} \leq 1$ (\cite[Theorem 1]{BLMN1}) and also for several curves with $\overline{\kappa} = 2$ (\cite[Theorem 2]{BLMN1}), in particular, for all known rational unicuspidal curves; see also the end of \S \ref{ss:basic} and Remark \ref{rem:kappa} for further details.

In \cite{BL1}, M. Borodzik and Ch. Livingston proved the conjecture for rational \emph{unicuspidal} curves in the general case (\cite[Theorem 1.1]{BL1}), thus obtaining a necessary combinatorial condition on the numerical invariants of local plane curve singularities occuring on rational unicuspidal curves.

Using this condition, in 2014 Tiankai Liu in his PhD thesis \cite[Theorem 1.1]{Liu} gave a full list of possible local types with two Newton pairs and asked which types on this list are indeed realizable by rational curves.

Actually, all but two types on his list are realizable, as it can be shown by results and constructions already present in the literature (Kashiwara \cite{Kashiwara}, Miyanishi and Sugie \cite{MiSu}, Fenske \cite{Fen}, Tono \cite{Tono1, Tono2}, Orevkov \cite{Or}).

\subsection{Acknowledgement} The author would like to thank Andr\'as N\'emethi for fruitful discussions and Maciej Borodzik for introducing Liu's work and emphasizing the significance of these observations.

\section{Preliminaries}

\subsection{Notation and definitions}

Let $h \in \mathbb{C}[x,y,z]$ be an irreducible homogeneous polynomial of degree $d$. Its zero set $C = \{ h=0 \} \subset \mathbb{C}P^2$ is called a \emph{complex projective curve}. A point $P \in C$ is called \emph{singular} if the derivative of the defining polynomial vanishes at that point, \textit{i.e.} $\frac{\partial h}{\partial x}|_P = \frac{\partial h}{\partial y}|_P = \frac{\partial h}{\partial z}|_P = 0$. $C$ is called \emph{unicuspidal} if it has only one singular point and at that point the local plane curve singularity is locally irreducible. $C$ is called \emph{rational} if its normalization is homeomorphic to $S^2$ (in the case of rational cuspidal curves, that is, rational curves having locally irreducible singularities only, $C$ itself is already homeomorphic to $S^2$).

To characterize the \emph{local embedded topological type} of a plane curve singularity, several invariants can be used, which are equivalent to each other. We are referring to \cite{BK, EN, Wall} for further information. Also, \cite[\S 2.1]{BLMN1} is a brief introduction to this topic. We recall the main facts in the following paragraphs.

Let $f: (\mathbb{C}^2,\mathbf{0}) \rightarrow (\mathbb{C},0)$ be a germ of a local holomorphic function which is not smooth (\textit{i.e.} its derivative vanishes at $\mathbf{0}$), irreducible in $\mathbb{C}\{x,y\}$. Then it has a \emph{local parametrization}, \textit{i.e.}  there exists $x(t), y(t) \in \mathbb{C}\{t\}$ such that $f(x(t),y(t)) \equiv 0$ and $t \mapsto (x(t), y(t))$ is a bijection for $|t| < \varepsilon$ small enough. Up to local homeomorphism, we can assume that the parametrization has the form
\begin{equation}
x(t) = t^a, \quad y(t) = t^{b_1} + \dots + t^{b_r}
\end{equation}
with $a < b_1 < \dots < b_r$, $a > \textrm{gcd}(a,b_1) > \textrm{gcd}(a,b_1,b_2) > \dots > \textrm{gcd}(a,b_1,b_2,\dots, b_r) = 1$. If $r = 2$, we say that the singularity has \emph{two Newton pairs}. More precisely, we say that a local plane curve singularity $\{ f(x,y) = 0 \}$ has two Newton pairs $(p_1,q_1)(p_2,q_2)$ if after a local homeomorphism the singularity can be parametrized by
\[ x(t) = t^{p_1p_2}, \quad  y(t) = t^{q_1p_2} + t^{q_1p_2 + q_2}. \]
In this case we simply say that the singularity \emph{is of type} $(p_1,q_1)(p_2,q_2)$, where $\textrm{gcd}(p_1,q_1) = \textrm{gcd}(p_2,q_2) = 1, p_1, p_2 \geq 2$ and, by convention, $p_1 < q_1$. Sometimes we will allow $p_1 = 1$, in such case we say that the Newton pair `degenerates' to one Puiseux pair $(p_2, p_2q_1 + q_2)$.

For any locally irreducible plane curve singularity $f(x,y)=0$ with parametrization $x(t), y(t)$ as above, one can take the set of all possible local intersection multiplicities, \textit{i.e.} the set
\[ \Gamma = \{ \textrm{ord}_t F(x(t), y(t)) : f \nmid F \in \mathbb{C}\{x,y\} \} \subset \mathbb{Z}_{\geq 0}, \]
which is easily seen to be an additive semigroup. Moreover, it is a cofinite set, that is, $|\mathbb{Z}_{\geq 0} \setminus \Gamma| = \delta$ and $\textrm{max} \{\mathbb{Z}_{\geq 0} \setminus \Gamma\} = 2\delta - 1$, where $\delta$ is the so called \emph{delta invariant} of the singularity.

We define the \emph{semigroup counting function} $R$ by
\[ R(n) = \# \Gamma \cap [0, n], \]
\textit{i.e.} $R(n)$ is the number of semigroup elements less than or equal to $n$.

\subsection{Basic facts and further notation}\label{ss:basic}

The property checked in \cite{BLMN1} for all known rational unicuspidal curves and finally proved in \cite{BL1} is the following:

\begin{theorem}\cite[Theorem 1.1]{BL1}
Let $R$ be the semigroup counting function of the only singular point of a rational unicuspidal curve of degree $d$ in $\mathbb{C}P^2$. Then the following condition must hold:
\begin{equation}\labelpar{eq:sdp}
  R(jd) = (j+1)(j+2)/2
\end{equation}
for every $j = 0, 1, \dots, d-3$.
\end{theorem}

In the rest of this subsection, we recall some results regarding the numerical invariants of plane curves. We rely mostly on \cite{BLMN2, BLMN1}.

In the case of a local singularity of type $(p_1, q_1)(p_2, q_2)$, the delta-invariant can be expressed as
\[ \delta = (p_1q_1p_2^2 + p_2q_2 - p_1p_2 - q_1p_2 - q_2 + 1)/2. \]
Also, the additive semigroup $\Gamma$ is generated (over $\mathbb{Z}_{\geq 0}$) by the following three elements $g_0, g_1, g_2$:
\begin{equation}\label{eq:semgen}
g_0 = p_1p_2, \quad g_1 = q_1p_2, \quad g_2 = p_1p_2q_1 + q_2.
\end{equation}

There are several invariants guiding the classification of projective plane curves. One can take the strict transform $\overline{C}$ under the \emph{local embedded minimal good resolution} $X \rightarrow \mathbb{C}P^2$ of the singularities  (only one singularity in our case) of $C$. (That is, we blow up $\mathbb{C}P^2$ several times until we resolve the singularities and obtain a normal crossing configuration: the exceptional divisors and the strict transform of the curve, which is smooth, intersect each other transversally and no three of them goes through the same point.) We denote by $\overline{C}^2$ the self-intersection of this strict transform $\overline{C}$ in $X$.

For a unicuspidal rational curve with two Newton pairs:
\begin{equation}
 \overline{C}^2 = d^2 - p_2q_2 -p_1q_1p_2^2.
\end{equation}

This, using the \emph{degree genus formula} for rational unicuspidal curves (that is, $(d-1)(d-2) = 2\delta$), turns into the following simpler expression:
\begin{equation}
 \overline{C}^2 = 3d - 1 - p_1p_2 - q_1p_2 - q_2.
\end{equation}

Denote by $\overline{\kappa}$ the \emph{logarithmic Kodaira dimension} of $\mathbb{C}P^2 \setminus C$ (see \cite{Iitaka, KM}). This turns out to be an extremely important invariant of a plane curve. Based on \cite[\S 1, (a), (b), (c)]{BLMN1}, recall the following. By the work \cite{Tsu} of Tsunoda, $\overline{\kappa} \neq 0$, therefore, the possibilities are $\overline{\kappa} \in \{ -\infty, 1, 2 \}$. Notice that in the case of unicuspidal curves, by the results of Yoshihara \cite{Yoshi}, $\overline{\kappa} = -\infty$ is equivalent with $\overline{C}^2 \geq -1$. Recall that all rational curves with $\overline{\kappa} = -\infty$ are classified by Kashiwara in \cite{Kashiwara}; see also \cite{MiSu}. Also, rational unicuspidal curves with $\overline{\kappa} = 1$ are classified by Tono in \cite{Tono1, Tono2}. This gives us a guiding principle where to look for the description of a unicuspidal curve with given numerical invariants.

Further facts and useful observations can be found in \cite{BLMN2, BLMN1}. 


\begin{remark}\label{rem:kappa}
 We wish to emphasize that since in \cite{BLMN1} the property \eqref{eq:sdp} was checked \emph{for all rational cuspidal curves with} $\overline{\kappa} \leq 1$, \emph{all the possible local types of cusps on such curves are explicitly listed in} \cite{BLMN1}. Moreover, the only known possible local types with two Newton pairs on a rational unicuspidal curve with $\overline{\kappa} = 2$ are those occuring on Orevkov's curves from \cite{Or} ((\ref{evii}) and (\ref{eviii}) in Theorem \ref{thm:main} below). For those, \eqref{eq:sdp} was also checked in \cite[Theorem 2 (c)]{BLMN1}. In particular, \emph{all the local types from Theorem \ref{thm:main} are already explicitly listed in \cite{BLMN1} with the appropriate references to their original constructions.}
\end{remark}

\section{The classification}

\subsection{List of possible types}

Based on the list in \cite[Theorem 1.1]{Liu}, we give the following complete list of singularity types. Set $F_0=0, F_1=1, F_{k+1}=F_k+F_{k-1}, k\geq 1$ for the Fibonacci numbers.

\begin{theorem}\label{thm:main}
 Let $C$ be a rational unicuspidal curve of degree $d$ whose singularity can be characterized by two Newton pairs. Then its type $(p_1,q_1)(p_2,q_2)$ is present in the list below. Conversely, for any member $(p_1,q_1)(p_2,q_2)$ of this list, there exists a rational unicuspidal curve $C$ with this type of singularity.
 
\begin{enumerate}[(i)]
  \item \label{ei} $(lF_{2k-1}^2+F_{2k-3}^2, lF_{2k+1}^2+F_{2k-1}^2+2)(F_{2k-1}^2, lF_{2k-1}^2+F_{2k-3}^2)$, 
        $d = F_{2k-1}F_{2k+1}(lF_{2k-1}^2+F_{2k-3}^2)$, $k \geq 2, l \geq 0$. ($k=2, l=0$ degenerates to one Puiseux type.)
  \item \label{eii} $(lF_{2k-1}^2+F_{2k-3}^2, lF_{2k+1}^2+F_{2k-1}^2+2)(F_{2k-1}, lF_{2k-1}+F_{2k-5})$, 
        $d = F_{2k+1}(lF_{2k-1}^2+F_{2k-3}^2)$, $k \geq 3, l \geq 0$.
  \item \label{eiii} $(n-1, n)(m, nm-1)$, $d = nm$, $n \geq 3, m \geq 2$.
  \item \label{eiv} $(n, 4n-1)(m, nm-1)$, $d = 2nm$, $n, m \geq 2$.
  \item \label{ev} $(n-1, n)(n, (n+1)^2)$, $d = n^2+1$, $n \geq 3$.
  \item \label{evi} $(n, 4n+1)(4n+1, (2n+1)^2)$, $d = 8n^2+4n+1$, $n \geq 2$.
  \item \label{evii} $(F_{4k}/3, F_{4k+4}/3)(3, 1)$, $d = F_{4k+2}$, $k \geq 2$.
  \item \label{eviii} $(F_{4k}/3, F_{4k+4}/3)(6, 1)$, $d = 2F_{4k+2}$, $k \geq 2$.
\end{enumerate}
\end{theorem}

\begin{proof}
First we show that being in the above list is necessary for the realizability. Assume that $(p_1, q_1)(p_2, q_2)$ is the local cusp type of a unicuspidal rational curve of degree $d$. Then by \cite[Theorem 1.1]{BL1} \eqref{eq:sdp} must hold. In this case, by \cite[Theorem 1.1]{Liu}, the Newton pair either equals to $(2,7)(4,17)$ with $d=17$ or to $(2,3)(6,31)$ with $d=20$, or it is present in the above list. $(2,7)(4,17)$ is excluded in \cite[\S 6.10]{BLMN2} by the \emph{spectrum semicontinuity} (SS) criterion (SS fails at $l = 12$ with the notations therein). The other pair $(2,3)(6,31)$ can be excluded similarly, now SS fails at $l = 13$.

Now we show that any element of the above list is realizable.

In case (\ref{ei}) and (\ref{eii}) one computes that $\overline{C}^2 = 0$ and $-1$, respectively. Therefore, in both cases, $\overline{\kappa} = -\infty$. In \cite{BLMN1}, based on \cite{Kashiwara} (cf. also \cite{MiSu}), a list of numerical invariants was presented for these curves. One finds that curves from (\ref{ei}) and (\ref{eii}) are exactly those listed in \cite[\S 6.2.1, \S 6.2.2 with $N=1$]{BLMN1}, respectively. 

The existence of the case (\ref{eiii}), as also Liu notices, follows from \cite[Theorem 1.1, 1a]{Fen} ($\overline{C}^2 = m$, $\overline{\kappa} = -\infty$), cf. also \cite[\S 4, (9)]{BLMN1}.

The existence of (\ref{eiv}) is proved in \cite[Theorem 1.1, (iv)]{Tono1} (set $k_1=1, d_1=nm, k_2=n-1, d_2=m, k_3=m-1$ to match the two notations). These are of Abhyankar--Moh type (cf. again \cite[\S 4]{BLMN1}), that is, the tangent line to its singular point has no other intersection with the curve; as it is easily seen by comparing the semigroup generators \eqref{eq:semgen} with the degree of the curves and then using B\'ezout's theorem. We have $\overline{C}^2 = m$ and $\overline{\kappa} = -\infty$.

For the existence of (\ref{ev}) (with $\overline{C}^2 = -n+1$) and (\ref{evi}) (with $\overline{C}^2 = -n$), see \cite[Theorem 2, (i), Type I, $s=2$ and (ii), Type II]{Tono2} (cf. also \cite[\S 7.1]{BLMN1}). In these cases, $\overline{\kappa} = 1$.

The existence of (\ref{evii}) and (\ref{eviii}) follows from \cite[Theorem C, b), c)]{Or}, respectively (cf. \cite[\S 9.1, \S 9.2]{BLMN1}). In these cases, one has $\overline{C}^2 = -2$ and $\overline{\kappa} = 2$. (Notice that Orevkov uses a \emph{characteristic sequence} rather than the Newton pairs; for the definition and comparison, see \cite[\S 3]{Or}).

\end{proof}

\end{document}